\documentclass[a4paper, reqno, 11pt]{amsart}
\usepackage{paper_preamble}
\usepackage{import}

\newcommand{\sgn}{\mathrm{sgn}}
\renewcommand{\IJ}{\operatorname{IJ}}
\DeclareMathOperator{\J}{J}

\DeclareMathOperator{\Adj}{Adj}

\renewcommand{\p}{\mathbb P}
\renewcommand{\phi}{\varphi}
\newcommand{\na}{\mathrm{Cr}}
\newcommand{\aut}{\mathrm{Aut}}
\newcommand{\pic}{\mathrm{Pic}}
\newcommand{\kk}{\mathbf k}

\begin{document}
    \title{Equivariant rationality of Fano threefolds in the family \textnumero 2.12}
	\author{Oliver Li, Joseph Malbon and Antoine Pinardin}
	\begin{abstract}
		We prove that a faithful group action on the smooth complete intersection $X$ of three divisors of bidegree $(1,1)$ in $\p^3\times\p^3$ is linearisable if and only if $\rk(\pic^G(X))\ne1$.
	\end{abstract}
	\maketitle
    \section{Introduction}
    All the work presented in this note is performed over an algebraicaly closed field $\kk$ of characteristic zero. 
    \subsection{$G$-rational Fano threefolds}
    Let $X$ be a rational variety of dimension $n$. Given a subgroup $G\subset\aut(X)$, a classical problem is to ask if $X$ is \textit{$G$-rational}, that is, if there exists a $G$-birational map $\phi\colon X\dashrightarrow\p^n$. This is equivalent to the condition that (for any birational identification of $X$ and $\P^n$) the group $G$ is conjugate in $\na_n(\kk)$ to a subgroup of $\aut(\p^n)$. In this case, we will also say that the group $G$ (or more precisely its action on $X$) is \textit{linearisable}. We are interested in the $G$-rationality of Fano threefolds. When $G$ stands for the Galois group of a non-closed field, this question has been fully answered by Kuznetsov-Prokhorov \cite{kuznetsov2019rationality}. They use a generalization to non-closed fields by Benoist-Wittenberg \cite{benoist2019clemens} of the intermediate Jacobian obstruction provided by Clemens and Griffiths in \cite{clemens1972intermediate}. In the present note, we utilise the Clemens-Griffiths criterion in a $G$-equivariant setting for geometric group actions over an algebraically closed field. We apply this technique for Fano threefolds in the deformation Family $\textnumero 2.12$. See \cite{hassett2021equivariant} and \cite{ciurca2024intermediate} for other applications of this method.

    \subsection{Fano threefolds in the family \textnumero 2.12}
    A Fano threefold $X$ in family \textnumero 2.12 is the smooth complete intersection of three divisors of bidegree $(1,1)$ in $\P^3 \times \P^3$. We may write the three equations in the form
    \begin{equation}\label{eqn: eqns}
        \mathbf{x}^T M_i \mathbf{y} = 0,
    \end{equation}where $M_i$ is a four-by-four matrix, $i$ varies from $1$ to 3, and $\mathbf{x} = [x_0:\dots:x_3]$ and $\mathbf{y} = [y_0:\dots:y_3] $ are the coordinates of a respective copy of $\P^3$.

    Now consider $\pi\colon X \rightarrow \P^3$, the first projection map. This map is the blowup along a non-hyperelliptic curve $C \subseteq \P^3$ of genus 3 and degree 6. Similarly, the second projection $\pi'\colon X \rightarrow \P^3$ is the blowup of a second curve $C'\subseteq \P^3$. The curves are isomorphic because they are both isomorphic to the plane quartic \[\det(xM_1 + yM_2 +zM_3) = 0 \] where the $M_i$ are as in \eqref{eqn: eqns}.

    Now let $G$ be a group acting faithfully on $X$. Recall that $X$ is \textit{$G$-Fano} if $\Pic^G(X) = \Z[-K_X ]$. Our main result, answering a question of Cheltsov-Li-Ma'u-Pinardin \cite[Remark 6]{2-12}, is the following.
    \begin{theorem}\label{thm: not-equivariant}
        Let $G$ be a group acting faithfully on $X$. Then the $G$-action on $X$ is linearisable if and only if $X$ is not $G$-Fano. 
    \end{theorem}

    If $X$ is not $G$-Fano, then the result is almost trivial, so the content of the result is the converse direction. Our approach is to isolate an element $\sigma\in G$ which leaves only $K_X$ invariant, and compute the $\sigma$-action on the Lie algebra of the intermediate Jacobian $\IJ_X$ of $X$. We then show that this differs by a sign from the action on the Lie algebra of $\J_C$, the Jacobian of $C$, where the action on $C$ is induced in a natural way. Since $C$ is a plane quartic, this would imply the result.
    \subsection{Acknowledgements} 
    We are extremely grateful to Ivan Cheltsov for bringing us together, introducing us to the problem, for his guidance and for suggesting the use of the intermediate Jacobian. We would also like to thank the organisers of the 2024 conference in Algebraic Geometry held at the Matrix Institute in Victoria, Australia and the Sydney Mathematical Research Institute, where this work was initiated. OL is supported by the Australian Commonwealth Government and would like to thank his advisor Jack Hall for several helpful discussions and suggestions. He would further like to thank Dougal Davis, Peter McNamara and Muhammad Haris Rao for their help on the project.
    \section{On the geometry of the family \textnumero 2.12} 
     In this section, we recall what we need about the geometry of our threefold. The bulk of this section is obtained from \cite[\S 4.1.2]{CAG}, but the reader can also refer to \cite[\S 4]{2-12}. Let $C\subseteq \P^2$ be a plane quartic curve and $D$ a degree 2 divisor such that $H^0(\O_C(D)) = 0$. Then $K_C+D$ is very ample \cite{Homma}. Set $V := H^0(\O_C(K_C+D))$ and $W := H^0(\O_C(2K_C-D))$. Then by \cite[Theorem 4.1.4]{CAG}, there is a linear map \[\rho\colon V^\vee \otimes W^\vee \rightarrow H^0(\P^2, \O_{\P^2}(1))\] such that the following diagram commutes 
     \begin{equation}\label{eqn: huge-fkn-diagram}
         \begin{tikzcd}[row sep = huge]
         C \arrow[r, hook] \arrow[d, "|K_C+D| \times |2K_C-D|"] & \P^2 \arrow[r, "\rho^\vee"] & \P(V \otimes W) \arrow[d, "\mathrm{Adj}"] \\
         \P(V^\vee) \times \P(W^\vee) \arrow[rr, "\mathrm{Segre}"] && \P(V^\vee \otimes W^\vee).
     \end{tikzcd}
     \end{equation}
     Set $X := \P(\ker \rho)\cap  (\P(V^\vee) \times \P(W^\vee))\subseteq \P(V^\vee \otimes W^\vee)$.
    \begin{proposition}\label{prop: 3fold-facts}
        The following hold:
        \begin{enumerate}[label = \normalfont(\roman*)]
            \item\ The curve $C$ is exactly the vanishing locus \[C = \{\det\circ \rho^\vee = 0\} \subseteq \P^2. \] Moreover, $\rk \rho^\vee(x) \geq 3$ for all $x\in \P^2$.
            \item The projection $\pi\colon X \rightarrow \P(V^\vee)$ is the blowup along the curve $C$ embedded in $\P(V^\vee)$ via the complete linear system $|K_C + D|$. In particular, $X$ is a smooth Fano threefold, and is a complete intersection in $\P(V^\vee) \times \P(W^\vee)$.
            \item Conversely, every smooth complete intersection of three (1,1) divisors in $\P^3 \times \P^3$ is obtained this way. 
        \end{enumerate}
    \end{proposition}
     \begin{proof}
         The first part of (i) is \cite[Theorem 4.1.4]{CAG}, and the rank assertion is explained in \cite[\S 4.1.1]{CAG}. 
         
         To prove (ii), first recall that the adjugate of a linear map $\varphi\colon V_1 \rightarrow V_2$ between $d$-dimensional vector spaces is the induced map $\Adj(\varphi)\colon\wedge^{d-1}V_1 \rightarrow \wedge^{d-1}V_2$. We may think of this as a map $V_1^\vee \rightarrow V_2 ^\vee$ using the obvious isomorphism $\wedge^{d-1}V_i = V_i^\vee$. Note that $\Adj\varphi$ is full rank if and only if $\varphi$ is, and that $\rk \Adj \varphi = 1$ if and only if $\rk \varphi = d-1$ (in this case $\im \Adj \varphi = \ker \varphi^\vee$ as can easily be seen). Now chasing \eqref{eqn: huge-fkn-diagram}, we find that the embedding of $C$ in $\P(V^\vee)$ is given by $x \mapsto \ker(\rho^\vee(x)\colon V \rightarrow W^\vee)$. Note that $\rk \rho^\vee(x) = 3$ by (i). 

         Now we claim the image of $C$ in $\P(V^\vee)$ is exactly the locus of points $v \in \P(V^\vee)$ where $\rho(v, -)\colon W^\vee \rightarrow H^0(\O(1))$ is not surjective. Indeed, if $v$ is the image of $x\in C$, then the composition \[W^\vee \xrightarrow{\rho(v, -)} H^0(\O(1)) \xrightarrow{x} \kk \] is the zero map. But $x$ is nonzero by assumption, so the map is degenerate. Conversely, suppose the map $\rho(v, -)\colon W^\vee \rightarrow H^0(\O(1))$ is not surjective. Then $$\rho(v, -)^\vee \colon H^0(\O(1))^\vee \rightarrow W$$ is not injective. Note that $\P(\ker\rho(v, -)^\vee) \subseteq C$ by (i), so $\P(\ker\rho(v, -)^\vee)$ is a point say $x$, and by construction $v$ is the image of $x$. 
         
         Now let us choose coordinates $\mathbf{x} = [x_0:\dots:x_3]$ for $\P(V^\vee)$ and $\mathbf{y} = [y_0:\dots:y_3] $ for $\P(W^\vee)$. We may represent the map $V^\vee \rightarrow W\otimes H^0(\O(1))$ as a matrix \[M=\begin{pmatrix}
            L_{10} & L_{11} & L_{12} & L_{13}\\
            L_{20} & L_{21} & L_{22} & L_{23}\\
            L_{30} & L_{31} & L_{32} & L_{33}
        \end{pmatrix}, \] where the $L_{ij}$'s are linear functionals in $\bf{x}$. Let $f_i$, for $i = 0,1,2,3$, denote the determinant of the submatrix of $M$ obtained by deleting the $i$-th column. Then $C$ (which we now identify with its image) is exactly the vanishing locus of these three cubics. 

        Let us now show that the preimage of this $C$ is a Cartier divisor. First note that given a general $[x_0:\cdots :x_3]\in \P(V^\vee)$, its fibre in $X$ is the point $([x_0:\cdots:x_1], [f_0:-f_1:f_2:-f_3])$, as can easily be checked. Now the question is local and since $C$ is smooth, it is (possibly after a change-of-variables) locally cut out by two of the $f_i$, say $f_0$ and $f_1$. In particular, we may write $f_2 = a_0 f_0 + a_1f_1$ and $f_3 = b_0 f_0 + b_1f_1$ where $a_i, b_i$ are local functions on $\P^3$, which also implies that $y_2 = a_0y_0 - a_1y_1$ and $y_3 = -b_0y_0 +b_1y_1$ and further that $f_0y_1+y_1f_0=0$. Thus on the locus $y_0 \neq 0$, we see that $f_1 =0$ cuts out the preimage. This shows the preimage of $C$ is a Cartier divisor. Thus by the universal property of blowups and the rank claim, the variety $X$ maps bijectively (and hence isomorphically, by Zariski's main theorem) to $\mathrm{Bl}_C\P(V^\vee)$ as claimed\footnote{Alternatively, one can observe that equation $f_0y_1+y_1f_0 = 0$ cuts out a variety isomorphic to the blowup along $C$ locally, and check that these local isomorphisms glue}. This proves (ii).

        Finally, given such a smooth complete intersection $X$ in $\P^3 \times \P^3$, we may write the three equations as in \eqref{eqn: eqns} and form the matrix $M = (L_{ij})$ as above. Then the projection $X\rightarrow \P^3$ is birational, and the indeterminacy locus of its inverse is given by the vanishing of the $f_i$ as above. Since $X$ is smooth, this indeterminacy locus is also smooth. In particular, for any $\mathbf{x}$ in this locus, we have $\rk M(\mathbf{x}) = 2$, since the singular locus is precisely the locus of points with $\rk < 2$ \cite[Proposition 4.2.2]{CAG}. Thus its fibre in $X$ would be $ \P(\ker M(\mathbf{x})) \cong \P^1$. In particular, this indeterminacy locus is of pure dimension 1 (i.e. a curve), which we shall denote $C$. Finally, according to \cite[\S7.2.2]{CAG}, the Hilbert polynomial of $C$ is $6t-2$ - in particular $C$ is a genus 3 curve.

        Now consider the determinant quartic $C' := \{\det(xM_1 + yM_2 + zM_3) = 0\}$ in $\P^2$. Let us define a map $C \rightarrow C'$ given by $$\mathbf{x} \mapsto \ker M(\mathbf{x})^\vee.$$ This is birational with inverse \[[x:y:z]\mapsto \ker(xM_1 + yM_2 + zM_3), \] and hence identifies $C$ with $\tilde{C'}$, the normalisation of $C'$. But $C'$ is already arithmetic genus 3, and hence must be normal already. One then checks that the variety obtained from the plane quartic $C = C'$ and the degree 2 divisor $\O_C(1) - K_C$ using the recipe at the beginning of this section is exactly the $X$ we started with. This concludes the proof.
     \end{proof}

    \section{Proof of Theorem \ref{thm: not-equivariant}}
    Denote $E$ denote the class of the exceptional divisor of the blowup $\pi\colon X\rightarrow \P^3$ and let $H$ be the pullback of the hyperplane class on $\P^3$. Then $\Pic(X) = \Z H \oplus \Z E$. Similarly, we have a second exceptional surface $E'$ and a second hyperplane class $H'$ coming from $\pi' \colon X \rightarrow \P^3$. Now $$H' \sim 3H-E, E' \sim 8H-3E,$$ and moreover $H$ and $H'$ span the nef cone. Note that $$K_X \sim -4H+E \sim -H-H',$$ in particular the embedding in $\P^3 \times \P^3$ (or more precisely its composition with the Segre embedding in $\P^{15}$) is anticanonical.
     We begin with an equivalent condition for $G$-Fanoness. 
    \begin{lemma}\label{lemma: G-Fano_criterion}
        Let $G$ be a group acting faithfully on $X$. Then $X$ is $G$-Fano if and only if $G$ contains an element $\sigma$ that swaps $H$ and $H'$.
    \end{lemma}
    \begin{proof}
        Since $H$ and $H'$ span the nef cone, this $G$ must preserve the set $\{H, H'\}$. If $G$ does not contain such a $\sigma$, then $G$ must preserve $H$ and $H'$ individually, whence $H, H'\in \Pic^G(X)$ so $\Pic^G(X) = \Z [H] \oplus \Z [H']$. Conversely, if $G$ contains such a $\sigma$, then $\Pic^G(X) = \Z[ -K_X] = \Z[H+H']$.        
    \end{proof}

    Let us assume $X$ is $G$-Fano, and let $\sigma\in G$ be an element which swaps $H$ and $H'$. It clearly suffices to prove (the nontrivial direction of) Theorem \ref{thm: not-equivariant} with $G = \langle \sigma \rangle$, so we assume this going forward.

    Now by assumption, we may choose an isomorphism $\alpha\colon\Ox(\sigma^* H) \iso \Ox(H')$. This in turn induces an isomorphism $$\beta\colon \Ox(\sigma^* H') = \Ox(\sigma^*(-K_X)) \otimes\Ox( \sigma^*(-H)) \iso \Ox(-K_X) \otimes \Ox( -H') = \Ox(H).$$ Then $\beta \circ\alpha$ (or more precisely $\beta \circ \sigma^*\alpha$) is a linearisation of the $\sigma^2$-action on $\Ox(H)$, and this induces a representation of $\langle \sigma^2 \rangle$ on $H^0(X, \Ox(H)) $, which splits as a direct sum of characters (i.e. is diagonalisable). Picking coordinates $x_0,\dots,x_3$ on $H^0(X, \Ox(H))$ such that $\sigma^2$ acts as  $$\sigma^2\cdot(x_0,\dots ,x_3) = (\varpi^{r_0} x_0,\dots ,\varpi^{r_3} x_3),$$ where $\varpi$ is a primitive $|\sigma|$-root of unity,  set $y_i := \alpha(x_i)\in H^0(X, \Ox(H'))$. Then $G$ acts on $H^0(X, \Ox(H)) \oplus H^0(X, \Ox(H'))$ as 
    \begin{equation}\label{eqn: g-action}
        \sigma\cdot ((x_0,\dots ,x_3), (y_0,\dots ,y_3)) = ( (y_0,\dots ,y_3),  (\varpi^{r_0} x_0,\dots ,\varpi^{r_3} x_3))  
    \end{equation}
    And this induces an action of $G$ on $\P^3 \times \P^3$ for which the embedding $X \hookrightarrow \P^3 \times \P^3$ is equivariant.

    This also induces a $G$-action on the two-dimensional linear system of $(1,1)$-divisors on $\P^3 \times \P^3$ which contain $X$. Explicitly, it is given by 
    \begin{equation}\label{eqn: g-action-p2}
        \sigma \cdot \mathbf{x}^TM_i \mathbf{y} = \mathbf{x}^T \diag( \varpi ^{r_0},\dots,\varpi^{r_3}) M_i^T \mathbf{y}, 
    \end{equation}
    where the $M_i$ are as in \eqref{eqn: eqns}.
    \begin{lemma}\label{lemma: g-invariant-divisors}
        If $X$ is $G$-Fano, then we may choose our three $(1,1)$-divisors to be $G$-invariant.
    \end{lemma}
    \begin{proof}
        The three equations (\ref{eqn: eqns}) cutting out $X$ is exactly the kernel of the restriction map $H^0(\P^3 \times \P^3, \O(1,1))\rightarrow H^0(X, \Ox(1,1))$ which is clearly equivariant. Since $G$ is cyclic, its action on $\ker(H^0(\P^3 \times \P^3, \O(1,1))\rightarrow H^0(X, \Ox(1,1)))$ is diagonalisable, and choosing an invariant basis, we obtain the result.
    \end{proof}
    \subsection{The $G$-action on $C$ and $\mathrm{J}_C$} \label{subsection: G-action-X}
    Let us now turn our attention to the $G$-action on $C$. By Lemma \ref{lemma: g-invariant-divisors}, we may choose our three divisors $F_i :=\mathbf{x}^TM_i\mathbf{y}=0$ to be $G$-invariant, say $\sigma$ acts on $F_i$ with eigenvalue $\varpi^{s_i}$. Identifying the two-dimensional linear system of $(1,1)$-divisors containing $X$ with $\P^2_{x,y,z}$ using this basis as in Proposition \ref{prop: 3fold-facts} and identifying $C$ with the plane quartic  \[ \det(xM_1 + yM_2 + zM_3) = 0,\] this in turn induces a $G$-action on $C$, hence a $G$-representation structure on $$\Lie(\mathrm{J}_C) = H^0(C, K_C).$$ Write $\chi_{r} $ for the character $G \rightarrow \Gm$ sending $\sigma$ to $\varpi^r$.
    \begin{proposition}
        We have \[\Lie(\mathrm{J}_C) = \bigoplus_{j = 1}^3\chi_{s_j +(\sum_{i = 1}^3 s_i) - (\sum_{i = 0}^3r_i) }. \]
    \end{proposition}
    \begin{proof}
        See also the proof of \cite[Theorem 3.3]{singular_cubic_vanya}. By adjunction, we have $H^0(C, K_C) = H^0(\P^2, K_{\P^2}(C)).$ Set $u := y/x$ and $v :=  z/x$. Then the two-forms \[\frac{du\wedge dv}{\det(M_1 + uM_2 + vM_2)}, u\frac{du\wedge dv}{\det(M_1 + uM_2 + vM_2)}, v\frac{du\wedge dv}{\det(M_1 + uM_2 + vM_2)} \] form a basis for $H^0(\P^2, K_{\P^2}(C))$.

        Let us compute the $G$-action on these forms. Since the $G$-action on $\P^2$ is induced by \eqref{eqn: g-action-p2}, the automorphism $\sigma$ acts on $\det(xM_1 + yM_2 + zM_3)$ with eigenvalue $\varpi^{\sum r_i}$. It therefore acts on $$\det(M_1 + uM_2 + vM_3) = \frac{\det(xM_1 + yM_2 + zM_3)}{x^4}$$ with eigenvalue $\varpi^{-4s_1+\sum r_i}$. Moreover we have $$\sigma^*(du \wedge dv) = \varpi^{-2s_1 + s_2 + s_3} du \wedge dv. $$ Combining, it follows 
        \begin{align*}
            \sigma^*\left(\frac{du\wedge dv}{\det(M_1 + uM_2 + vM_2)}\right) &= \varpi^{2s_1 + s_2 + s_3 - \sum r_i},\\
            \sigma^*\left(u\frac{du\wedge dv}{\det(M_1 + uM_2 + vM_2)}\right) &= \varpi^{s_1 + 2s_2 + s_3 - \sum r_i},\\
            \sigma^*\left(v\frac{du\wedge dv}{\det(M_1 + uM_2 + vM_2)}\right) &= \varpi^{s_1 + s_2 + 2s_3 - \sum r_i}
        \end{align*}
        as claimed.
    \end{proof}

    \subsection{The action on the intermediate Jacobian}
    The key technical result is the following:
    \begin{proposition}\label{prop: ij}
        Let $\IJ_X$ be the intermediate Jacobian of $X$. Then $$ \Lie( \IJ_X) = \bigoplus_{j = 1}^3\chi_{s_j +(\sum_{i = 1}^3 s_i) - (\sum_{i = 0}^3 r_i)}\otimes \sgn,$$ where $\sgn: G \rightarrow \Gm$ is the sign representation which sends $\sigma$ to $-1$. In particular, the three-dimensional $G$-representations $\Lie(\IJ_X)$ and $\Lie(\J_C)$ differ by a sign, and therefore cannot be isomorphic.
    \end{proposition}
    We proceed with a few lemmas:
    \begin{lemma}\label{lemma: truncation}
        Let $0 \rightarrow\F^{-n} \rightarrow\dots \rightarrow \F^0 \rightarrow \F^{1} \rightarrow 0$ be an exact sequence of sheaves on $X$. Let $0\leq s \leq n$. If for every $r\in \{0,...,n\}\setminus \{s\}$ we have $H^i(X, \F^{-r}) = 0$ for all $i\in \Z$ then we have $H^i(X, \F^{1}) = H^{i+s}(X, \F^{-s})$ for all $i\in \Z$.
    \end{lemma}
  
    \begin{proof}
        Let us induct on $n$. If $n=0$ the result is trivial. Now supposing true for $n-1$, we consider the exact sequence $0 \rightarrow\F^{-n} \rightarrow\dots \rightarrow \F^0 \rightarrow \F^{1} \rightarrow 0$. Write $d^i$ for the map $d^i\colon \F^i \rightarrow \F^{i+1}$. We may truncate the sequence to obtain the following exact sequences: 
        \begin{equation}\label{eqn: left-truncation}
            0 \rightarrow  \F^{-n}\rightarrow ... \rightarrow \F^{-s-1}\rightarrow \im d^{-s-1}\rightarrow 0,
        \end{equation}
        \begin{equation}\label{eqn: middle-truncation}
            0 \rightarrow \im d^{-s-1} = \ker d^{-s} \rightarrow \F^{-s}\rightarrow \im d^{-s}\rightarrow 0, 
        \end{equation}
        and 
        \begin{equation}\label{eqn: right-truncation}
            0 \rightarrow \im d^{-s} \rightarrow \F^{-s+1} \rightarrow ... \rightarrow \F^0 \rightarrow \F^1 \rightarrow 0.
        \end{equation}
        By induction applied to (\ref{eqn: left-truncation}), it follows that $\im d^{-s-1}$ has no cohomology in any degree. It thus follows from (\ref{eqn: middle-truncation}) that $H^i(X, \F^{-s}) = H^i(X, \im d^{-s})$ for all $i\in \Z$. Finally, it follows from the inductive hypothesis that $H^{i+s}(X, \im d^{-s}) = H^i(X, \F^1)$ for all $i\in \Z$, as desired.
    \end{proof}
    \begin{remark}
        This also follows from the $E_1$-hypercohomology spectral sequence $E_1^{p,q} = H^q(X, \F^{p}) \Rightarrow \bb{H}^{p+q}(X, \F^\bullet)$. See \cite[\href{https://stacks.math.columbia.edu/tag/012K}{Tag 012K}]{stax}.
    \end{remark}

    \begin{lemma}\label{lemma: vanishing-h2h3}
        We have \[H^2(X, \Omega^1_{\P^3 \times \P^3}|_X) = 0 = H^3(X, \Omega^1_{\P^3 \times \P^3}|_X). \]
    \end{lemma}
    \begin{proof}
        Recall the Euler sequence on $\P^3$: \[0 \rightarrow \Omega^1_{\P^3} \rightarrow \O_{\P^3}(-1)^4 \rightarrow \O_{\P^3} \rightarrow 0. \] Combining a copy of this sequence from each copy of $\P^3$ and restricting  to $X$, we obtain the following exact sequence on $X$: 
        \begin{equation}\label{eqn: euler}
            0 \rightarrow \Omega^1_{\P^3 \times \P^3}|_X\rightarrow \O_{X}(-1, 0)^4 \oplus \O_{X}(0, -1)^4 \rightarrow \O_{X}^2 \rightarrow 0. 
        \end{equation} Now since $X$ is a complete intersection in $\P^3 \times \P^3$, we may take its Koszul resolution: \begin{equation}\label{eqn: koszul}
            0 \rightarrow \O_{\P^3 \times \P^3}(-3, -3) \xrightarrow{d^{-3}} \O_{\P^3 \times \P^3}(-2, -2)^3 \xrightarrow{d^{-2}} \O_{\P^3 \times \P^3}(-1, -1)^3 \xrightarrow{d^{-1}} \O_{\P^3 \times \P^3} \xrightarrow{d^0} \Ox \rightarrow 0.
        \end{equation}
        More generally, by twisting (\ref{eqn: koszul}) we obtain exact sequences: 
        \begin{equation}\label{eqn: koszul-twisted}
            0 \rightarrow \O_{\P^3 \times \P^3}(-3+a, -3+b) \xrightarrow{d^{-3}_{a,b}} \dots \xrightarrow{d^{-1}_{a,b}} \O_{\P^3 \times \P^3}(a,b) \xrightarrow{d^0_{a,b}}\Ox(a,b) \rightarrow 0
        \end{equation}for each $a,b\in \Z$.
        Now applying Lemma \ref{lemma: truncation} to (\ref{eqn: koszul}) with $s = 0$ and $n = 3$ (noting that by the K\"unneth formula we have $H^i(\O_{\P^3\times \P^3}(-r, -r) = 0$ for all $1 \leq r \leq 3$) we find that 
        \begin{equation}\label{eqn: cohomology-ox}
           H^i(X, \Ox) =H^i(\O_{\P^3 \times \P^3}) = \begin{cases}
                \C & \text{if }i=0,\\
                0 &\text{otherwise}.
            \end{cases}
        \end{equation}
        Now applying Lemma \ref{lemma: truncation} to (\ref{eqn: koszul-twisted}) with $a = -1, b = 0$, we have 
        \begin{equation}\label{eqn: cohomology-ox(-1,0)}
            H^i(\Ox(-1, 0)) = 0 = H^i(\Ox(0, -1))
        \end{equation} for all $i\in \Z$ (the second equality holds by symmetry).
        
        The result now follows from (\ref{eqn: cohomology-ox}) and (\ref{eqn: cohomology-ox(-1,0)}) applied to the long exact sequence of (\ref{eqn: euler}).    
        
    \end{proof}
    
    \begin{proof}[Proof of Proposition \ref{prop: ij}]
        Recall that $\Lie(\IJ_X) = H^2(X, \Omega_X^1)$. This follows directly from the Hodge diamond.
        Let us consider the conormal sheaf sequence. \[0 \rightarrow \I/\I^2\rightarrow \Omega^1_{\P^3 \times \P^3}|_X \rightarrow \Omega^1_X \rightarrow 0. \] Since the map $\Omega^1_{\P^3 \times \P^3}|_X \rightarrow \Omega^1_X$ is clearly equivariant, the conormal sheaf $\I/\I^2$ naturally inherits a $G$-linearisation. In particular, each of its cohomologies inherit the structure of a $G$-representation. Thus by Lemma \ref{lemma: vanishing-h2h3}, there is an isomorphism of $G$-representations $H^3(X, \I/\I^2) \cong H^2(X, \Omega^1_X)$.
        
        By Lemma \ref{lemma: g-invariant-divisors}, we may choose our three (1,1) divisors to be $G$-invariant. Since $X$ is the complete intersection of these divisors, this means the conormal sheaf $\I/\I^2$ decomposes into three $G$-invariant line bundles, say $\I/\I^2 = \bigoplus_{i = 1}^3 \L_i$ each abstractly isomorphic to $\Ox(-1, -1)$ and each inheriting a linearisation determined by the $G$-action on $F_i = \mathbf{x}^T M_i \mathbf{y}$. In particular, on $\L_i$ which corresponds to $F_i$, the element $\sigma$ acts on a local generator with eigenvalue $\varpi^{s_i}$.

        Let us consider the twisted Koszul resolution (\ref{eqn: koszul-twisted}) with $a = b = -1$ again. Since our divisors are invariant, it is clear that for all linearisations of $\Ox(-1, -1)$, we can linearise $\O_{\P^3 \times \P^3}(r,r)$ for $-4 \leq r \leq -1$, so that the maps in (\ref{eqn: koszul-twisted}) are equivariant. In particular, applying Lemma \ref{lemma: truncation}, we find isomorphisms of $G$-representations $$H^3( \L_i) \cong H^6(\O_{\P^3 \times \P^3}(-4, -4))$$ for each $i$. Let us compute the $G$-action on $H^6(\O_{\P^3 \times \P^3}(-4, -4))$ for each $i$.

        Set $\eta_x := (x_0x_1x_2x_3)^{-1} \in H^3(\P^3, \O(-4))$, and similarly with $\eta_y$. By the K\"unneth formula \cite[\href{https://stacks.math.columbia.edu/tag/0BED}{Tag OBED}]{stax}, a basis element of $H^6(\O_{\P^3 \times \P^3}(-4, -4))$ is given by $\pi^*(\eta_x) \cup \pi'^*(\eta_y)$. Our $\sigma$ sends this element to $$\varpi^{-\sum r_i}\pi^*(\eta_y) \cup \pi'^*(\eta_x) = (-1)^{3\times 3}\varpi^{-\sum r_i}\pi^*(\eta_x) \cup \pi'^*(\eta_y) = -\varpi^{-\sum r_i}\pi^*(\eta_x) \cup \pi'^*(\eta_y).$$ However, since our $\O_{\P^3 \times \P^3}(-4, -4)$ is obtained from a Koszul resolution, it has local generator $F_1 \wedge F_2 \wedge F_3$. Our $\sigma$ acts on this local generator with eigenvalue $\varpi^{\sum s_i}$. However, the $\sigma$ also acts on a local generator of $\L_i$ with eigenvalue $\varpi^{s_j}$ as explained above. Thus in conclusion $\sigma$ acts on the one-dimensional $H^6(\O_{\P^3 \times \P^3}(-4, -4)) \cong H^3(\L_i)$ with eigenvalue $-\varpi^{s_j + (\sum_{i = 1}^3 s_i) - (\sum_{i = 0}^3 r_j)},$ as asserted.
    \end{proof}
    \begin{remark}
        In fact, the same calculation shows that if $\sigma$ does not swap $H$ and $H'$, then $\Lie \IJ_X$ and $\Lie \J_C$ are equivariantly isomorphic. Indeed, the sign comes from the swapping of $\eta_x$ and $\eta_y$ in the cup product.
    \end{remark}
    \begin{proof}[Proof of Theorem \ref{thm: not-equivariant}]
        Firstly recall that $\Aut(X) $ sits in an exact sequence \[1 \rightarrow \Aut(\P^3, C) \rightarrow \Aut(X) \rightarrow \mumu_2, \] where the final map is surjective if and only if $\Aut(X)$ contains an element which swaps $E$ and $E'$ - see \cite[\S 1]{2-12}. Now suppose $X$ is not $G$-Fano. Then $G$ is contained in $\Aut(\P^3, C) \subseteq \Aut(X)$, and thus the blow-down $\pi\colon X \rightarrow \P^3$ of $E$ is $G$-equivariant.
        
        We now suppose that $X$ is $G$-Fano. Then by Lemma \ref{lemma: G-Fano_criterion} $G$ contains an element $\sigma$ which swaps $H$ and $H'$. It clearly suffices to assume $G = \langle \sigma \rangle$. Now suppose that $\phi\colon X\dashrightarrow\P^3$ is a $G$-equivariant birational map, where $G$ acts linearly on $\P^3$. By \cite[Theorem 1.3.3]{abramovich2019functorial} (see also \cite[Prop. 3.6]{KreschAndrew2022Ebta}), $\phi$ factors equivariantly as:
        $$\phi=\varphi_l\circ\varphi_{l-1}\circ\dots\circ\varphi_1,$$
        where for each $i$, either $\varphi_i$ or $\varphi_i^{-1}$ is a blow-up of a smooth $G$-invariant centre. 
        
        Since $X$ is the blow-up of $\P^3$ along the curve $C $, then by \cite[Lemma 8.1.2.]{parshin1999algebraic} its intermediate Jacobian $\IJ_X$ may be identified with the Jacobian $\mathrm{J}_C$ of $C$. Thus since $X$ has Picard rank two, then by the Torelli theorem and the uniqueness the decomposition of principally polarised abelian varieties into simple components, we have that one of the $\varphi_i$ is a blowup along a $G$-invariant curve which is isomorphic to $C$. In particular, we have a $G$-action $\alpha_1\colon G \times C \rightarrow C$ such that the isomorphism $\IJ_X \cong \J_C$ is $G$-equivariant. Let us show this cannot be the case.

        As explained in \S \ref{subsection: G-action-X}, the $G$-action on $X$ induces a $G$-action $\alpha_2\colon G \times C\rightarrow C$ on $C$. By Proposition \ref{prop: ij}, the $G$-action on $\Lie(\IJ_X) = \Lie(\J_C)$ induced by $\alpha_1$ and the $G$-action induced by $\alpha_2$ differ by a sign and are not isomorphic. But since $C$ is a plane quartic, the action on $C$ is determined by the action on $\P(H^0(C, K_C)) = \P(\Lie(\J_C))$. In particular, $\alpha_1$ and $\alpha_2$ must be the same group action, which is absurd, since their induced actions on $H^0(C, K_C)$ do not agree.
    \end{proof}

	\bibliographystyle{alpha}
	\bibliography{bib}
\end{document}